\documentclass[letterpaper, 11pt]{amsart}

\usepackage{fullpage}
\usepackage{graphicx}
\usepackage{amsmath, amssymb, amsfonts, amsthm}
\usepackage{url}
\usepackage{enumitem}
\usepackage[foot]{amsaddr}

\usepackage{tikz} 

\usepackage[ruled,vlined]{algorithm2e}

\usepackage{color}
\definecolor{darkblue}{rgb}{0,0,0.4}
\usepackage[plainpages=false,colorlinks=true,citecolor=blue,filecolor=black,linkcolor=red,urlcolor=darkblue]{hyperref}

\newtheorem{thm}{Theorem}[section]
\newtheorem{cor}[thm]{Corollary}

\newtheorem{lem}[thm]{Lemma}
\theoremstyle{remark}

\newcommand{\N}{\mathbb N}

\graphicspath{{../figs/}} 

\usepackage[backend=biber,giveninits=true,sorting=nyt,style=alphabetic,natbib=true,maxcitenames=2,maxbibnames=10,url=false,doi=true,backref=false]{biblatex}
\renewbibmacro{in:}{\ifentrytype{article}{}{\printtext{\bibstring{in}\intitlepunct}}}

\begin{document}
\title{Steklov eigenvalues of nearly spherical domains}

\author{Robert Viator}
\address{Swathmore College, Swathmore, PA}
\email{rviator1@swarthmore.edu}

\author{Braxton Osting}
\address{Department of Mathematics, University of Utah, Salt Lake City, UT}
\email{osting@math.utah.edu}
\thanks{B. Osting acknowledges partial support from NSF DMS 17-52202}

\subjclass[2010]{
35C20, 
35P05, 
41A58} 

\keywords{Steklov eigenvalues, perturbation theory, isoperimetric inequality}

\date{\today}

\begin{abstract} 
We consider Steklov eigenvalues of three-dimensional, nearly-spherical domains. In previous work, we have shown that the Steklov eigenvalues are analytic functions of the domain perturbation parameter. Here, we compute the first-order term of the asymptotic expansion, which can explicitly be written in terms of the Wigner 3-$j$ symbols. We analyze the asymptotic expansion and prove the isoperimetric result that, if $\ell$ is a square integer, the volume-normalized $\ell$-th Steklov eigenvalue is stationary for a ball. 
\end{abstract}

\maketitle

\section{Introduction} 
Let $\Omega \subset \mathbb R^d$ and consider the Steklov eigenproblem on $\Omega$,
\begin{subequations} \label{e:Steklov}
\begin{align}
\label{e:Steklova}
\Delta u &= 0  && \textrm{in } \Omega \\ 
\label{e:Steklovb}
\partial_n u &= \lambda u && \textrm{on } \partial \Omega.
\end{align}
\end{subequations}
Here $\Delta$ is the Laplacian acting on $H^1 (\Omega)$, and $\partial_n$ denotes the outward normal derivative on the boundary, $\partial \Omega$.  It is a well-known fact that, when $\partial \Omega$ is smooth, the Steklov spectrum is discrete, and the eigenvalues can be enumerated in increasing order, $0 = \lambda_0 (\Omega) < \lambda_1(\Omega) \leq \lambda_2(\Omega) \dots$, where $\lambda_n(\Omega) \rightarrow \infty$ as $n\rightarrow \infty$.  For a more general description of the Steklov spectrum, see \cite{girouard2014spectral}.

Isoperimetric inequalities for non-trivial Steklov eigenvalues have been explored since the mid-twentieth century.  
The first major result in the shape-optimization of Steklov eigenvalues were obtained by R. Weinstock \cite{weinstock1954inequalities}, which showed that the disc is the shape in $\mathbb{R}^2$ with largest first non-zero Steklov eigenvalue among all smooth bounded domains of fixed area.  This result was extended to $\mathbb{R}^d$ for $d\geq 3$ by F. Brock \cite{brock2001isoperimetric}.  
Bogosel, Bucur, and Giacomini \cite{bogosel2017optimal} obtained general existence results for shape optimizers for general Steklov eigenvalues $\lambda_n(\Omega), n\geq2$.  
The Steklov eigenvalue maximization problem for fixed perimeter has been studied numerically in two dimensions \cite{Akhmetgaliyev2016}
and three and four dimensions \cite{Antunes2021}. 
Tuning of mixed Steklov-Neumann boundary conditions have also been recently studied by Ammari, Imeri, and Nigam \cite{nigam2020}, where an algorithm was designed to generate the proper mixed boundary conditions necessary to obtain desired resonance effects.

Steklov eigenvalues have applications in electromagnetism and materials design \cite{Lipton_1998-Archive}, \cite{Lipton_1998}.  Recently they have been used in a nondestructive testing method to locate defects in a medium using measured far-field data \cite{Cakoni_2016}.  For this problem, numerical results reveal that a localized defect of the refractive index in a disc perturbs only a small number of Steklov eigenvalues. 

\subsection*{Results}
Let $\Omega = \Omega_\varepsilon$ be a nearly-spherical domain where the boundary can be expressed in spherical coordinates (radius $r$, inclination $\theta\in[0,\pi]$, azimuth $\phi\in[0,2\pi]$) and expanded in the basis of real spherical harmonics, 
\begin{equation} \label{e:dom}
\Omega_\varepsilon=\{(r,\theta,\phi)\colon  0\le r\le 1 + \varepsilon \rho(\theta,\phi)\}, 
\qquad\text{where} \ \ \rho(\theta,\phi)=\sum_{\ell=0}^{\infty}\sum_{m = -\ell}^\ell A_{\ell,m} Y_{\ell,m}(\theta,\phi)
\end{equation}
is a given $C^1(\partial \Omega_0)$ \emph{perturbation function}. 

For $\varepsilon = 0$, $\Omega_0$ is the unit ball and the eigenvalues are $\lambda_{\ell,m} = \ell$ (multiplicity $2\ell +1$) with corresponding eigenfunctions 
$$
u_{\ell,m} (r,\theta,\phi) = r^\ell Y_{\ell,m}(\theta,\phi), \qquad \qquad \ell \in \mathbb N, \ |m|\leq \ell.
$$ 

In previous work \cite{Viator2019}, we have shown that
$\lambda = \lambda(\varepsilon)$ 
is analytic with respect to $\varepsilon$.
The method of proof is to treat such domains as perturbations of the ball, we prove the analyticity of the Dirichlet-to-Neumann operator  with respect to the domain perturbation parameter. Consequently, the Steklov eigenvalues are also shown to be analytic in the domain perturbation parameter \cite{Kato_1966}. 

The \emph{goal of this paper} is to obtain and study the first term of the asymptotic expansion of 
$\lambda = \lambda(\varepsilon)$  
in terms of the small parameter, $\varepsilon$. 
This extends the work in \cite{Viator_2018}, where the same problem is studied in dimension two for reflection-symmetric domains.  We will then use the asymptotic expansion to obtain local optimizers for isoperimetric inequalities for certain Steklov eigenvalues.

In particular, in Section~\ref{s:EigPert}, we derive an asymptotic expansion for Steklov eigenvalues satisfying \eqref{e:Steklov} for a domain $\Omega_\varepsilon$ of the form \eqref{e:dom} for small perturbation parameter $\varepsilon >0$. 
For $k\in \mathbb N$, consider the group of eigenvalues $\{\lambda_n(\varepsilon)\}_ {k^2}^{(k+1)^2-1}$, which satisfy $\lambda_n(0) = k$. 
In Theorem~\ref{t:StekAsym}, we characterize the first-order behavior in $\varepsilon$, {\it i.e.}, find the first term in the expansion  
$$
\lambda_{n}(\varepsilon) = k + \varepsilon \lambda_{n}^{(1)} + O (\varepsilon^2).
$$
We show that the perturbation of these $2k+1$ eigenvalues $\lambda_{n}(\varepsilon)$ are described by the eigenvalues of a real, symmetric $2k+1 \times 2k+1$ matrix, denoted $M^{(k)}$, whose entries are given by
$$
M^{(k)}_{m,n} = - \frac{1}{2} \sum_{p=0}^{\infty}\sum_{q = -p}^p A_{p,q}  \left( p(p+1) + 2k \right)  \iint \limits_{S^2} Y_{p,q}(\theta,\phi) Y_{k,m}(\theta,\phi) Y_{k,n}(\theta,\phi)  \ dS . 
$$
Interestingly, due to the non-simplicity of the eigenvalues, the eigenvalues are not Fr\'echet differentiable at $\varepsilon = 0$. This is manifested in the fact that the first-order eigenvalue perturbation is not described in terms of a linear functional, but rather the eigenvalues of a finite matrix. This should be contrasted with the reflection-symmetric two-dimensional case, where the symmetry can be used to decompose the eigenspaces into dimension one subspaces and the first-order perturbation can be written as a linear functional of the perturbation coefficients \cite{Viator_2018}. 

Further interpretation of the asymptotic results (Theorem~\ref{t:StekAsym}) are given in Corollary~\ref{c:OneCoef} in the case where only one spherical harmonic is perturbed. In particular, we show that high frequency oscillations in the domain not perturb low eigenvalues.  This result is consistent with the behavior seen in \cite{Cakoni_2016}, although the eigenvalue problem and the nature of the perturbation they consider is different; the authors consider a \emph{material} perturbation, while we consider shape deformation.

In Section~\ref{s:stationary}, we further analyze the matrix $M^{(k)}$ to prove the following isoperimetric result. Denote the volume-normalized Steklov eigenvalue by $\Lambda_\ell(\Omega ) := \lambda_\ell(\Omega) \cdot |\Omega|^{\frac 1 3}$. 
\begin{thm} \label{t:stationary}
Let $k \in \mathbb N$. Then $\Lambda_{k^2}$ is stationary for a ball in the sense that, for every perturbation function $\rho$, the map $\varepsilon \mapsto \Lambda_{k^2}(\Omega_\varepsilon)$ is non-increasing in $|\varepsilon|$ for $|\varepsilon|$ sufficiently small. 
\end{thm}
Theorem~\ref{t:stationary} suggests that for $\ell$ a squared integer, $\Lambda_\ell(\Omega)$ is possibly maximized when $\Omega$ is a ball. However, recent numerical results suggest that $\Lambda_4$ is maximized by the ball while $\Lambda_9$ is not maximized by the ball \cite{Antunes2021}. 

From the eigenvalues of the ball, $\Omega_0$, one can easily see that $\sum_{n=-k}^k \lambda_{k,n}(\Omega_0) = (2k+1)k$. It follows from the proof of Theorem~\ref{t:stationary} that 
\begin{equation} \label{e:partialTrace}
\sum_{\ell = k^2}^{(k+1)^2 - 1} \lambda_\ell(\Omega_\varepsilon) = (2k+1)k + O(\varepsilon^2), 
\end{equation}
meaning that the sum of a grouping of eigenvalues is invariant to perturbation at first-order.  

\begin{figure}[ht!]
\begin{center}
\includegraphics[width=.40\textwidth]{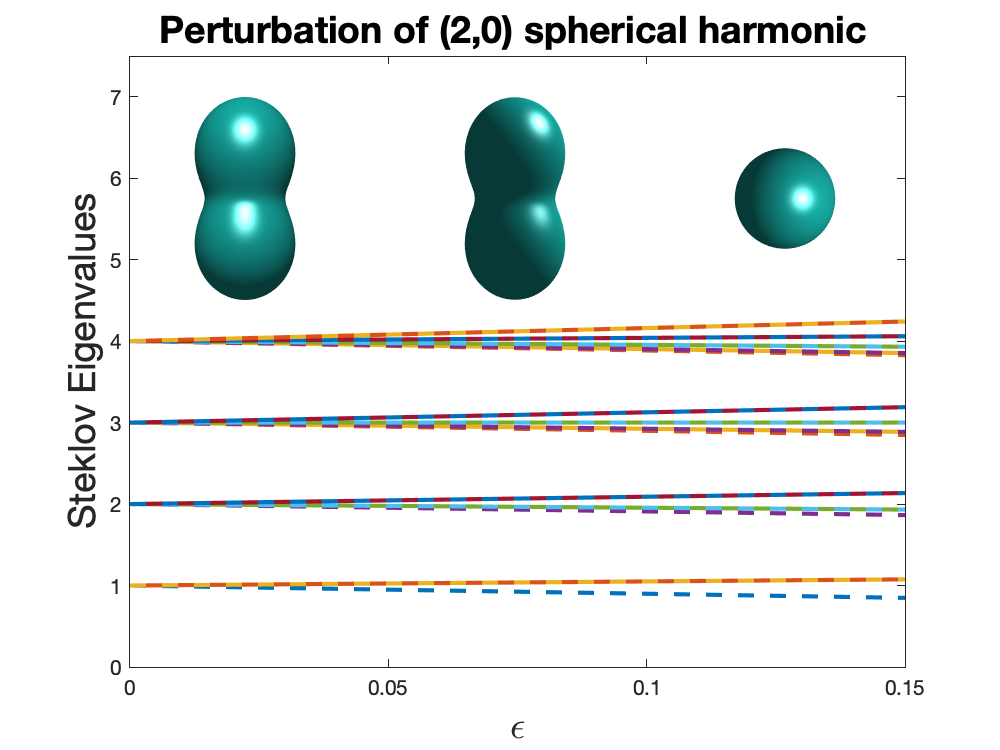}
\includegraphics[width=.40\textwidth]{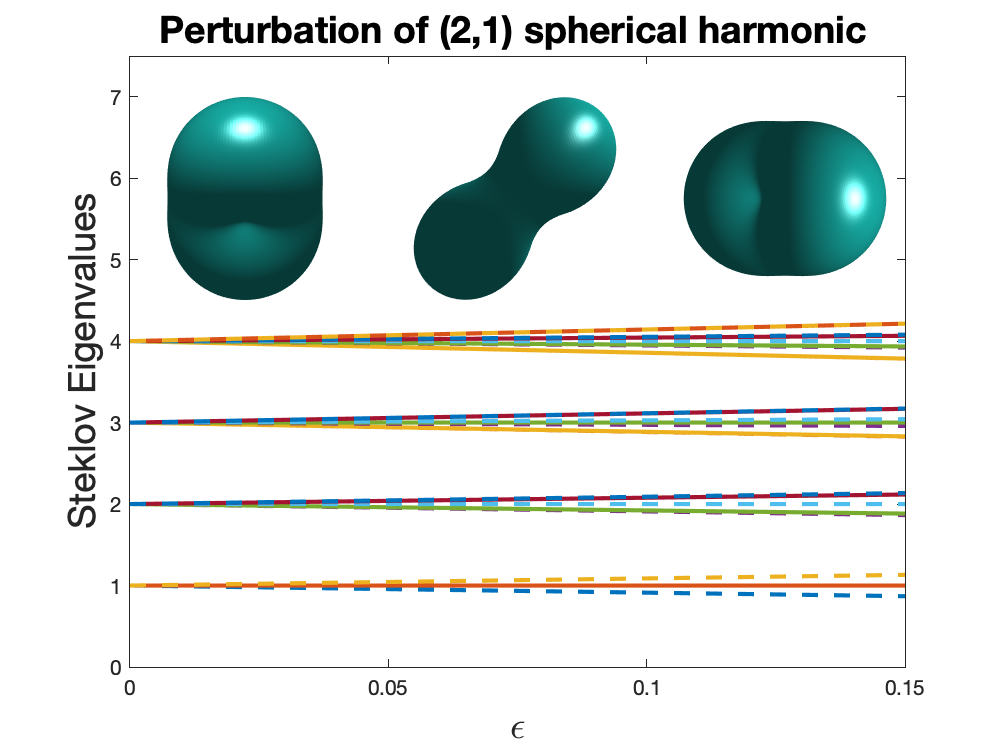}
\includegraphics[width=.40\textwidth]{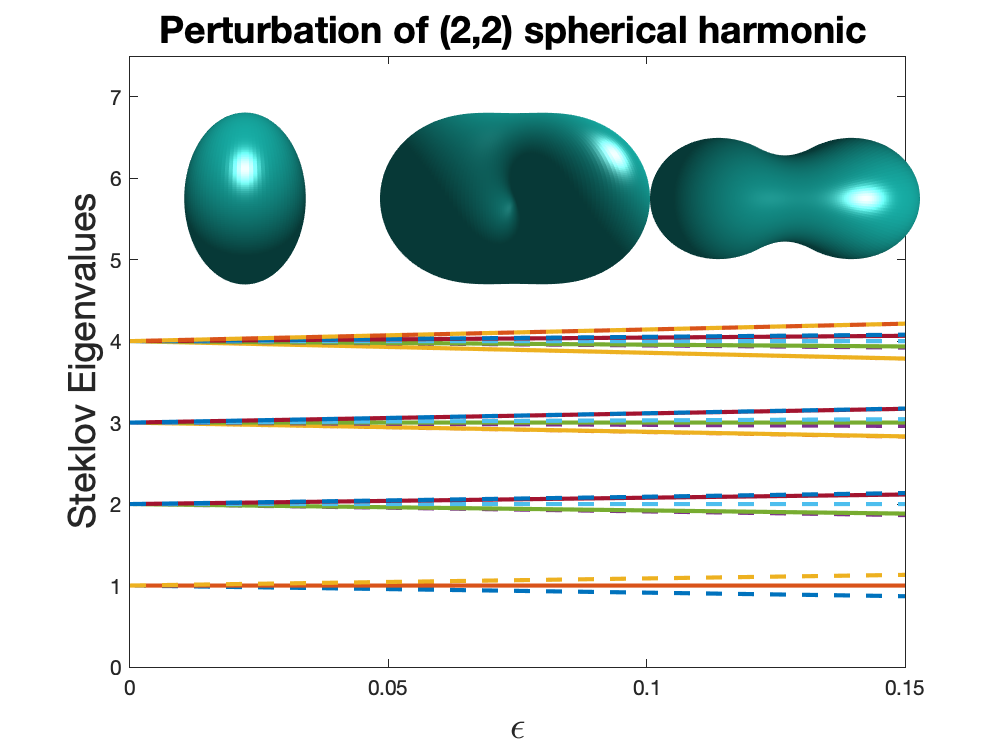}
\includegraphics[width=.40\textwidth]{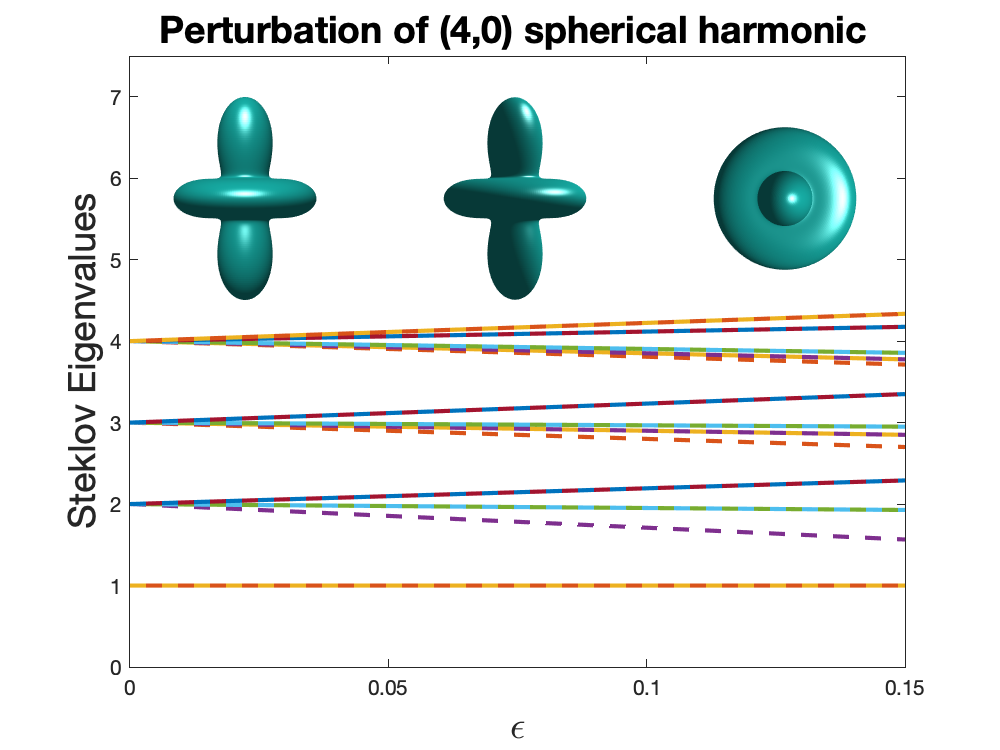}
\includegraphics[width=.40\textwidth]{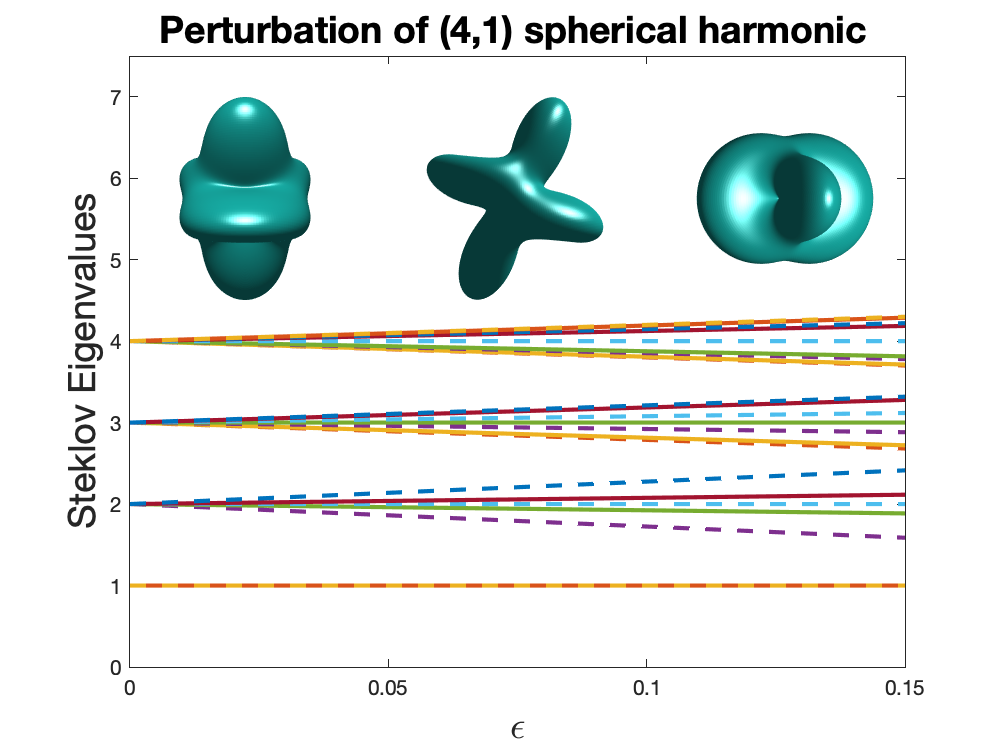}
\includegraphics[width=.40\textwidth]{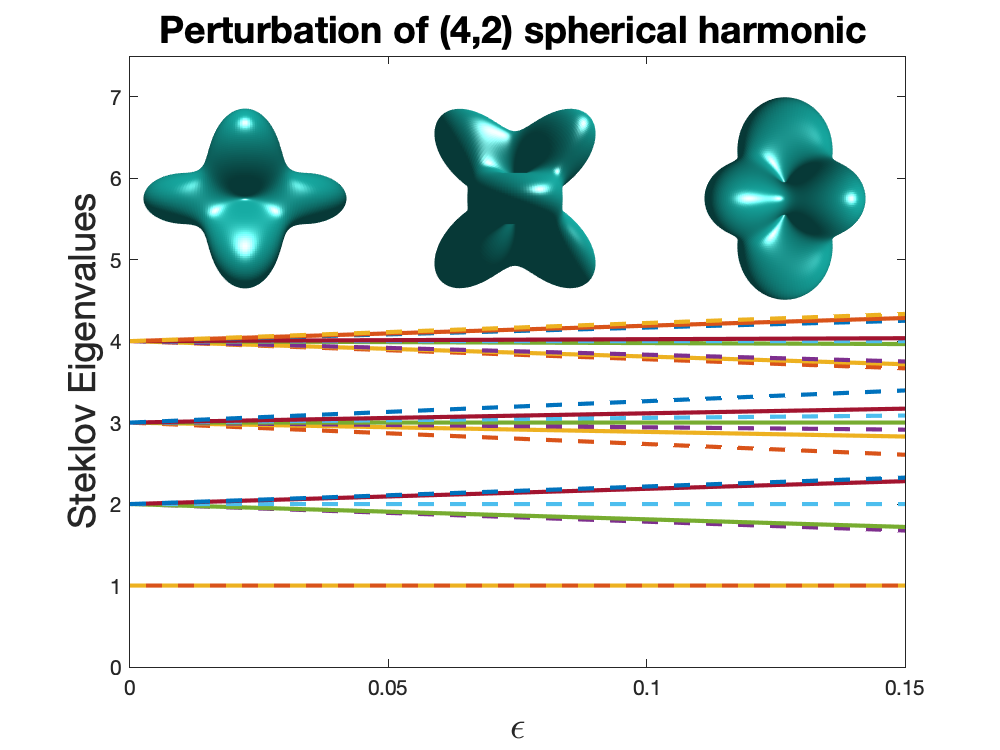}
\includegraphics[width=.40\textwidth]{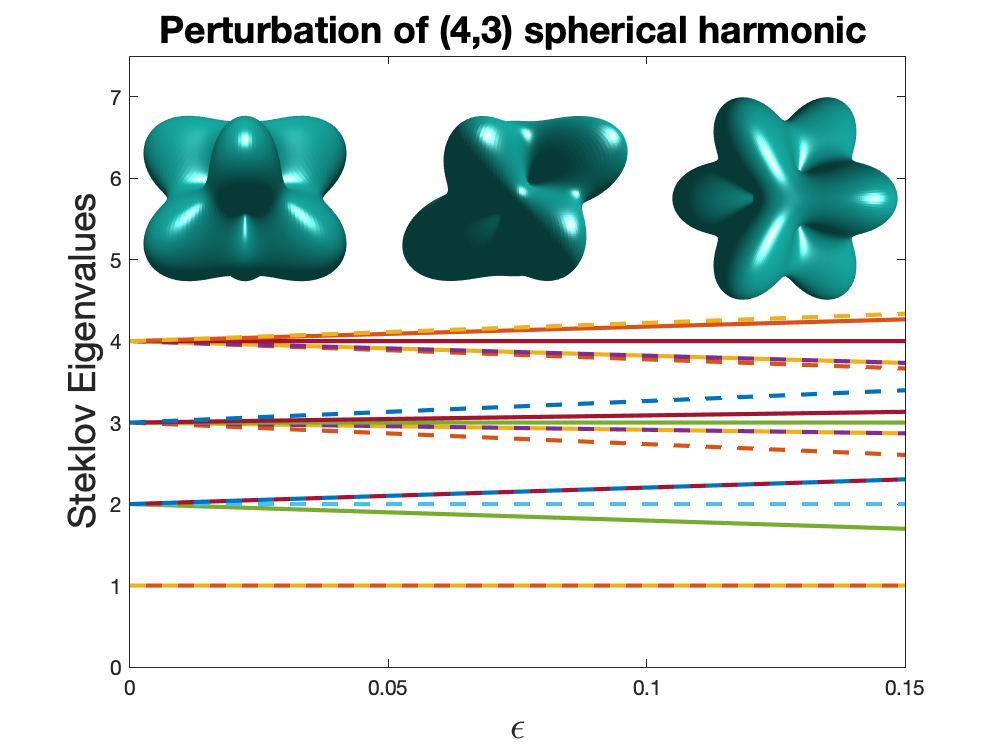}
\includegraphics[width=.40\textwidth]{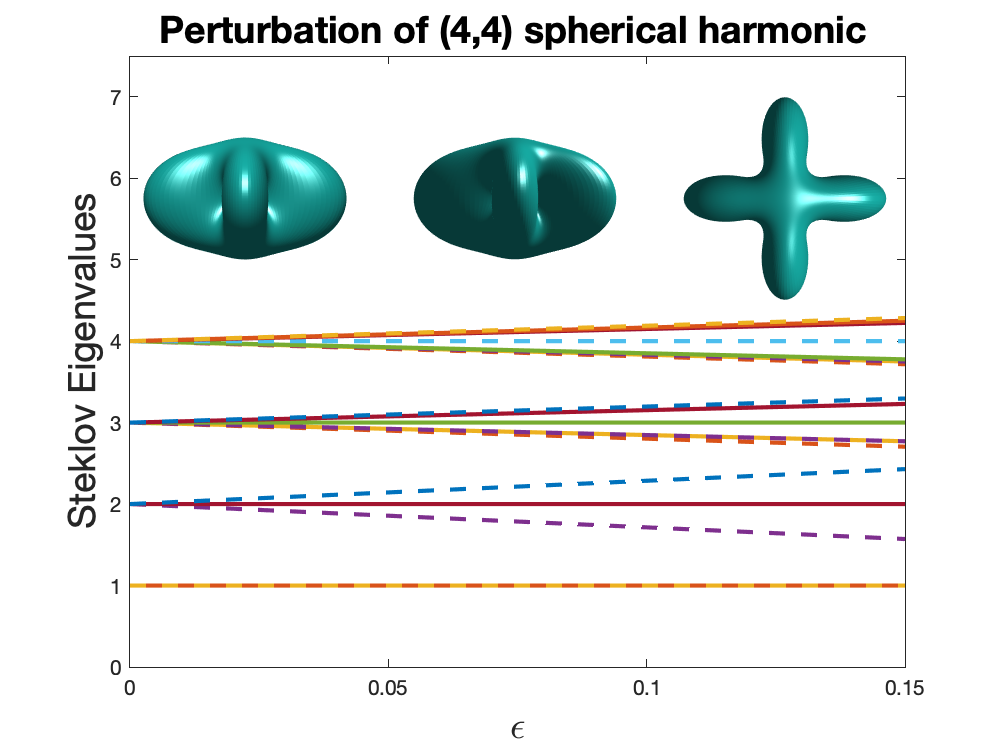}
\caption{A plot of the first-order approximation for Steklov eigenvalues satisfying \eqref{e:Steklov} on a domain of the form \eqref{e:dom} with $\rho(\theta,\phi) = 1 + \varepsilon Y_{p,q}(\theta,\phi)$ for indicated $(p,q)$ spherical harmonic. See Section~\ref{s:AsymSteklov}.}
\label{f:PerturbBall}
\end{center}
\end{figure}

Our main results are illustrated in Figure~\ref{f:PerturbBall}. Here, for various choices of $(p,q)$, we show the  first-order approximation for Steklov eigenvalues for a domain of the form \eqref{e:dom} with $\rho(\theta,\phi) = 1 + \varepsilon Y_{p,q}(\theta,\phi)$. We observe that, although we have \eqref{e:partialTrace} is satisfied, this is not always due to the direct cancellation of pairs of eigenvalues. For example, for $(p,q)=(2,0)$, the first eigenvalue group splits into three eigenvalues: a multiplicity two eigenvalue which is positively perturbed and a simple  eigenvalue which is negatively perturbed. The magnitudes of the perturbations are such that the sum is zero.


\section{An Asymptotic expansion for Steklov eigenvalues of nearly-spherical domains} \label{s:EigPert}
In this section, we derive an asymptotic expansion for Steklov eigenvalues satisfying \eqref{e:Steklov} for a domain $\Omega_\varepsilon$ of the form \eqref{e:dom} for small perturbation parameter $\varepsilon >0$. We recall that the real spherical harmonics, $Y_{\ell,m}$, in \eqref{e:dom} can be obtained from the complex spherical harmonics as follows. Define the \emph{complex spherical harmonic} by 
 \begin{equation} \label{e:ComplexHarmonic}
 Y_{\ell}^{m}(\theta,\phi)  =  \sqrt{\frac{(2\ell+1)}{4\pi}\frac{(\ell-m)!}{(\ell+m)!}}P_{\ell}^{m}(\cos(\theta))e^{im\phi}, 
 \qquad \qquad 
 \ell \geq 0, \ \  |m| \leq \ell,  
\end{equation}
where $P_{\ell}^{m}$ is the \emph{associated Legendre polynomial}, which can be defined through the Rodrigues formula, 
$P_{\ell}^{m}(x) = \frac{(-1)^{m}}{2^{\ell}\ell!}(1-x^{2})^{\frac{m}{2}}\frac{d^{m+\ell}}{dx^{m+\ell}}(x^{2}-1)^{\ell}$.
For $\ell \geq 0$ and $|m| \leq \ell$, the  \emph{real spherical harmonics} are then defined by
\begin{align}
\label{e:RealHarmonicsa}
Y_{\ell,m}(\theta,\phi) &=
\begin{cases}
	\frac{i}{\sqrt{2}} \left[ Y_{\ell}^m(\theta,\phi) - (-1)^mY_{\ell}^{-m}(\theta,\phi) \right] & \text{if}\ m<0 \\
	Y_{\ell}^0(\theta,\phi) & \text{if}\ m=0  \\
	\frac{1}{\sqrt{2}} \left[ Y_{\ell}^{-m}(\theta,\phi) + (-1)^mY_{\ell}^{m}(\theta,\phi) \right] & \text{if}\ m>0 
\end{cases} 
\end{align}

We begin by deriving asymptotic expansions for geometric quantities associated with $\Omega_\varepsilon$, in particular the volume of $\Omega_\varepsilon$ and the outward normal vector to $\partial \Omega_\varepsilon$. 

\subsection{Asymptotic expansions for geometric quantities}

\subsubsection{Volume}
Denoting the measure $dS = \sin \theta d\phi d\theta$, the volume of $\Omega_\varepsilon$ can easily be computed as 
\begin{align*}
|\Omega_\varepsilon| &= \frac{1}{3} \iint_{S^2} (1 + \varepsilon \rho(\theta,\phi))^3 \  dS \\
&= \frac{4 \pi}{3} + \varepsilon  \iint_{S^2}  \rho(\theta,\phi) \ dS + O (\varepsilon^2) \\
&= \frac{4 \pi}{3} + \varepsilon A_{0,0} \iint_{S^2}  Y_{0,0}(\theta,\phi)   \ dS + O (\varepsilon^2) \\
&= \frac{4 \pi}{3} + \varepsilon \sqrt{4 \pi} A_{0,0}  + O (\varepsilon^2) 
\end{align*}

\subsubsection{Normal vector}
It is convenient to denote the spherical coordinate vectors
$$
\hat r = \begin{pmatrix}
\sin(\theta) \cos (\phi) \\ \sin(\theta) \sin(\phi) \\ \cos(\theta)
\end{pmatrix}, 
\qquad
\hat \theta = \begin{pmatrix}
\cos(\theta) \cos (\phi) \\ \cos(\theta) \sin(\phi) \\ -\sin(\theta)
\end{pmatrix}, 
\quad \textrm{and} \quad
\hat \phi = \begin{pmatrix}
- \sin(\phi) \\ \cos(\phi) \\ 0
\end{pmatrix}. 
$$
The boundary can then be expressed as $x(\theta,\phi) = \left(1 + \varepsilon \rho(\theta,\phi)\right) \hat r$. 
The outward unit normal to the boundary of the domain can be computed by
$\hat{n}_\varepsilon(\theta,\phi) = \frac{ x_\theta \times x_\phi}{ |x_\theta \times x_\phi|}$. We compute
\begin{align*}
x_\theta &= \varepsilon \rho_\theta \hat r + \left(1 + \varepsilon \rho \right) \hat \theta \\
x_\phi & = \varepsilon \rho_\phi \hat r + \left(1 + \varepsilon \rho \right) \sin (\theta) \hat \phi.
\end{align*} 
Using the relationships 
$\hat \phi \times \hat r = \hat \theta$, 
$\hat \theta \times \hat \phi = \hat r$, and 
$\hat r \times \hat \theta = \hat \phi$, 
we obtain a (non-normalized) vector that is outward normal to the boundary, 
\begin{align*}
\tilde{n}_\varepsilon(\theta,\phi) &= x_\theta \times x_\phi \\
&= (1+ \varepsilon \rho) \left( (1 + \varepsilon \rho) \sin \theta \hat r - \varepsilon \rho_\theta \sin \theta \hat \theta - \varepsilon \rho_\phi \hat \phi \right) \\
&= \left( \left(1 + 2 \varepsilon \rho \right) \hat r - \varepsilon \rho_\theta \hat \theta \right) \sin (\theta) - \varepsilon \rho_\phi \hat \phi + O(\varepsilon^2) . 
\end{align*}
We compute
$$
|\tilde{n}_\varepsilon(\theta,\phi) |^{-1} = \frac{1}{\sin(\theta)} \left( 1 - 2 \varepsilon \rho \right) + O(\varepsilon^2).  
$$
The unit-normalized outward normal vector is then 
\begin{align}
\label{normalizednormal}
\hat{n}_\varepsilon =  | \tilde{n}_\varepsilon |^{-1} \tilde{n}_\varepsilon 
=  \vec{n}_0 + \varepsilon \vec{n}_1 +  O(\varepsilon^2), 
\end{align}
where
\begin{subequations}
\label{normalterms}
\begin{align}
\vec{n}_0 & =  \hat{r}\\
\vec{n}_1 & =  - \left( \rho_\theta \hat \theta + \frac{\rho_\phi}{ \sin(\theta) } \hat \phi \right) 
\end{align}
\end{subequations}

\subsection{Perturbation of eigenvalues}
We consider the perturbation of a Steklov eigenpair, 
\begin{align*}
& \lambda^\varepsilon =  \lambda^{(0)} + \varepsilon \lambda^{(1)} + O(\varepsilon^2)\\
& u^\varepsilon (r,\theta,\phi)  =  u^{(0)} (r,\theta,\phi)  + \varepsilon u^{(1)} (r,\theta,\phi)  + O(\varepsilon^2), 
\end{align*}
 due to a perturbation in the domain of the form in \eqref{e:dom}. 
 For fixed $k \in \mathbb N$, we let 
 \begin{subequations} \label{e:ZeroOrder}
 \begin{align} \label{e:ZeroOrdera}
 & \lambda_k^{(0)} = k\\
&  u_k^{(0)} (r,\theta,\phi) = \sum_{m=-k}^k  \alpha_m r^{k} Y_{k,m}(\theta, \phi). 
\end{align}
\end{subequations}
Note that we can't apriori determine the coefficients $\alpha_m$ that will select  the $O(1)$ eigenfunction  from the $2k+1$ dimensional eigenspace. 
Since the family $\{r^k Y_{k,n}(\theta, \phi)  \}_{k \in \mathbb{N}, \ |n|\leq k}$ forms a complete orthonormal basis for $L^2 (\Omega_\varepsilon)$  with $\varepsilon = 0$, we expand the higher-order eigenfunction perturbations in this basis. Thus, we make the following perturbation ansatz in $\varepsilon$ for the eigenvalue $\lambda_k^\varepsilon$ and corresponding eigenfunction  $u_k^\varepsilon$ 
\begin{subequations}
\label{e:expansionS}
\begin{align}
\lambda_k^\varepsilon & =  k + \varepsilon \lambda^{(1)} + O(\varepsilon^2) \\
u_k^\varepsilon(r,\theta,\phi) & = \sum_{\ell=0}^{\infty}\sum_{m = -\ell}^\ell  \Big( \delta_{\ell,k} \alpha_{m} + \varepsilon \beta_{\ell, m} + O(\varepsilon^2)  \Big) r^\ell Y_{\ell,m}(\theta, \phi). 
\end{align}
\end{subequations}
This ansatz satisfies \eqref{e:Steklova} exactly and we will determine the eigenvalue perturbation $\lambda^{(1)}$ and the coefficients $\alpha_m$ and $\beta_{\ell,m}$ so that \eqref{e:Steklovb} is satisfied.
Using the identity $\nabla = \partial_r \hat r + r^{-1} \partial_\theta \hat \theta + \frac{1}{r \sin(\theta)} \partial_\phi  \hat\phi$, we have that 
\begin{equation} \label{e:vj}
\nabla u_k^\varepsilon =  \sum_{\ell=0}^{\infty}\sum_{m = -\ell}^\ell   \Big( \delta_{\ell,k} \alpha_{m} + \varepsilon \beta_{\ell, m} + O(\varepsilon^2)  \Big)  r^{\ell-1} \vec{v}_{\ell,m}
\end{equation} 
where
$$
\vec{v}_{\ell,m} = \ell Y_{\ell,m}(\theta,\phi) \hat r + \partial_\theta Y_{\ell,m}(\theta,\phi) \hat \theta + \frac{1}{\sin(\theta) } \partial_\phi Y_{\ell,m}(\theta,\phi) \hat \phi.
$$
Denoting the expansion of the normal vector by 
$ \hat{n}_\varepsilon =  \vec{n}_0 + \varepsilon \vec{n}_1 + O(\varepsilon^2)$ as in \eqref{normalizednormal},  
we have the left hand side (LHS) and right hand side (RHS) of \eqref{e:Steklovb} are given by
\begin{subequations} \label{expHS}
\begin{align}\label{expLHS}
\textrm{LHS}  =    \sum_{\ell \in \mathbb{N}, \ |m|\leq \ell}  & \    
\left(  \delta_{\ell,k}  \alpha_{m} + \varepsilon \beta_{\ell,m} + O(\varepsilon^2) \right)   
\left((1 + \varepsilon (\ell-1) \rho + O(\varepsilon^2) \right) \ \left(\vec{n}_0 + \varepsilon \vec{n}_1 +O(\varepsilon^2) \right)\cdot  \vec{v}_{\ell,m}  \\
\label{expRHS}
\textrm{RHS}  =  \sum \limits_{\ell \in \mathbb{N}, \ |m|\leq \ell} &  \ 
\left(\lambda_k^{(0)} + \varepsilon \lambda_k^{(1)} + O(\varepsilon^2) \right)   
\left( \delta_{\ell,k}  \alpha_{m} + \varepsilon \beta_{\ell,m} + O(\varepsilon^2)  \right)  \ 
 \left(1 + \varepsilon \ell \rho + O(\varepsilon^2) \right) Y_{\ell,m}(\theta,\phi). 
\end{align}
\end{subequations}

Equating $O(\varepsilon^0)$  terms in \eqref{expLHS} and \eqref{expRHS}, we recover \eqref{e:ZeroOrdera}. 
Equating $O(\varepsilon^1)$  terms in \eqref{expLHS} and \eqref{expRHS}, we obtain 
\begin{align} \label{e:FirstO1eq}
\lambda_{k}^{(1)} \sum_{m=-k}^k \alpha_{m} Y_{k,m}(\theta, \phi) &= 
-  \sum_{m=-k}^k \alpha_{m}  \left( k \rho(\theta,\phi)Y_{k,m}(\theta,\phi) 
+ \rho_\theta \partial_\theta Y_{k,m}(\theta,\phi) 
+\frac{\rho_\phi}{ \sin^2(\theta)} \partial_\phi Y_{k,m}(\theta,\phi) \right)  \\
\nonumber
& \quad  - \sum_{\ell \in \mathbb{N}, \ |m|\leq \ell}   (k - \ell) \beta_{\ell,m} Y_{\ell,m}
\end{align}
Next we multiply by both sides of \eqref{e:FirstO1eq} by $Y_{k,n}(\theta, \phi)$ for $|n|\leq k$, integrate over $\Omega_0 = S^2$ with respect to the measure  $dS = \sin(\theta)  d\phi d\theta$, and use the orthogonality of the real spherical harmonics to obtain 
\begin{equation} \label{e:Lam1}
\lambda_{k}^{(1)} \alpha_n = 
 \sum_{m=-k}^k M^{(k)}_{m,n} \alpha_{m} 
 \qquad \implies \qquad 
  M^{(k)} \alpha =  \lambda_{k}^{(1)} \alpha 
\end{equation}
where the matrix $M^{(k)}\in \mathbb R^{ 2k+1 \times 2k+1}$ has entries given by 
$$
M^{(k)}_{m,n}  = 
- \iint  \left( k \rho(\theta,\phi)Y_{k,m}(\theta,\phi)  
+ \rho_\theta \left(\partial_\theta Y_{k,m}(\theta,\phi) \right) 
+\frac{\rho_\phi}{ \sin^2(\theta)} \left( \partial_\phi Y_{k,m}(\theta,\phi) \right) \right) Y_{k,n}(\theta,\phi)  \ dS  
$$
We now simplify $M^{(k)} $ as follows. 
Integrating by parts, the second term in $M^{(k)} $  can be written 
\begin{align*}
- \iint &  \rho_\theta \left(\partial_\theta Y_{k,m}(\theta,\phi) \right) Y_{k,n}(\theta,\phi)  \ dS \\
= & - \frac{1}{2} \iint   \rho_\theta \left[  Y_{k,n}(\theta,\phi)  \partial_\theta Y_{k,m}(\theta,\phi) -  Y_{k,m}(\theta,\phi)  \partial_\theta Y_{k,n}(\theta,\phi) \right] \ dS \\
 & +   \frac{1}{2}  \iint  \sin^{-1}(\theta) \partial_\theta \left( \sin(\theta) \rho_\theta \right)  Y_{k,m}(\theta,\phi) Y_{k,n}(\theta,\phi)  \ dS  \\
= & \  \frac{1}{2} \iint   \rho \left[  Y_{k,n}(\theta,\phi)  \partial^2_\theta Y_{k,m}(\theta,\phi) -  Y_{k,m}(\theta,\phi)  \partial^2_\theta Y_{k,n}(\theta,\phi) \right] \ dS \\
 & +   \frac{1}{2}  \iint  \sin^{-1}(\theta) \partial_\theta \left( \sin(\theta) \rho_\theta \right)  Y_{k,m}(\theta,\phi) Y_{k,n}(\theta,\phi)  \ dS   
\end{align*}
Similarly, the third term in $M^{(k)} $  can be written 
\begin{align*}
 - \iint & \frac{\rho_\phi}{ \sin^2(\theta)} \left( \partial_\phi Y_{k,m}(\theta,\phi) \right) Y_{k,n}(\theta,\phi)  \ dS  \\
= & - \frac{1}{2} \iint  \frac{\rho_\phi}{ \sin^2(\theta)} \left[ Y_{k,n}(\theta,\phi)    \partial_\phi Y_{k,m}(\theta,\phi)- Y_{k,m}(\theta,\phi) \partial_\phi Y_{k,n}(\theta,\phi) \right]  \ dS  \\
 & +   \frac{1}{2}  \iint  \frac{\partial_\phi^2 \rho}{ \sin^2(\theta)}  Y_{k,m}(\theta,\phi) Y_{k,n}(\theta,\phi)  \ dS  \\
= & \  \frac{1}{2} \iint  \rho \left[ Y_{k,n}(\theta,\phi)  \frac{ \partial^2_\phi Y_{k,m}(\theta,\phi)}{\sin^2(\theta)}  - Y_{k,m}(\theta,\phi) \frac{ \partial^2_\phi Y_{k,n}(\theta,\phi)}{\sin^2(\theta)}\right]  \ dS  \\
 & +   \frac{1}{2}  \iint  \frac{\partial_\phi^2 \rho}{ \sin^2(\theta)}  Y_{k,m}(\theta,\phi) Y_{k,n}(\theta,\phi)  \ dS  
\end{align*}
Denoting  the spherical Laplacian by
$ \Delta_S u = \frac{1}{\sin (\theta)} \partial_\theta \left( \sin (\theta) \ \partial_\theta u   \right) + \frac{1}{\sin^2 (\theta)} \partial_\phi^2 u$,
we obtain
\begin{align*}
M^{(k)}_{m,n}  = & 
- \frac{1}{2} \iint \left( - \Delta_S \rho + 2 k \rho  \right) Y_{k,m}(\theta,\phi) Y_{k,n}(\theta,\phi)  \ dS  \\
& - \frac{1}{2} \iint \rho \left( Y_{k,n}(\theta,\phi) \Delta_S Y_{k,m}(\theta,\phi)   -Y_{k,m}(\theta,\phi) \Delta_S Y_{k,n}(\theta,\phi) \right)  \ dS
\end{align*}
Using the fact that $-\Delta_S Y_{k,n} = k(k + 1) Y_{k,n}$ and writing 
$\rho(\theta,\phi)=\sum_{p=0}^{\infty}\sum_{q = -p}^p A_{p,q} Y_{p,q}(\theta,\phi)$, 
we have  
\begin{subequations}\label{e:M}
\begin{align} 
M^{(k)}_{m,n}  
&= - \frac{1}{2} \iint \left( -\Delta_S \rho  + 2 k \rho  \right) Y_{k,m}(\theta,\phi) Y_{k,n}(\theta,\phi)  \ dS   \\
&= - \frac{1}{2} \sum_{p=0}^{\infty}\sum_{q = -p}^p A_{p,q}  \iint \left( -\Delta_S Y_{p,q}(\theta,\phi)  + 2 k Y_{p,q}(\theta,\phi)  \right) Y_{k,m}(\theta,\phi) Y_{k,n}(\theta,\phi)  \ dS   \\
&= - \frac{1}{2} \sum_{p=0}^{\infty}\sum_{q = -p}^p A_{p,q}  \left( p(p+1) + 2k \right)  \iint Y_{p,q}(\theta,\phi) Y_{k,m}(\theta,\phi) Y_{k,n}(\theta,\phi)  \ dS   \\
&= - \frac{1}{2} \sum_{p=0}^{\infty}\sum_{q = -p}^p A_{p,q}  \left( p(p+1) + 2k \right)  W^{p,k}_{q,m,n}. 
\end{align}
\end{subequations}
where 
\begin{equation} \label{e:W}
W^{p,k}_{q,m,n} = \iint Y_{p,q}(\theta,\phi) Y_{k,m}(\theta,\phi) Y_{k,n}(\theta,\phi)  \ dS . 
\end{equation}

\subsection{Evaluation of $W^{p,k}_{q,m,n}$}
In \eqref{e:W}, we require the evaluation of $W^{p,k}_{q,m,n}$, the integral of the triple product of real spherical harmonic functions. 
Recall that the product of three \emph{complex} spherical harmonics can be expressed in terms of the Wigner 3-j symbol by
\begin{equation} \label{e:trip}
\iint Y_{\ell_1}^{m_1}Y_{\ell_2}^{m_2}Y_{\ell_3}^{m_3} \,dS  
  =
\sqrt{\frac{(2\ell_1+1)(2\ell_2+1)(2\ell_3+1)}{4\pi}}
\begin{pmatrix}
  \ell_1 & \ell_2 & \ell_3 \\
  0 & 0 & 0
\end{pmatrix}
\begin{pmatrix}
  \ell_1 & \ell_2 &\ell_3\\
  m_1 & m_2 & m_3
\end{pmatrix}.
\end{equation}
The Wigner 3-j symbol, $\begin{pmatrix}
  \ell_1 & \ell_2 &\ell_3\\
  m_1 & m_2 & m_3
\end{pmatrix}$, is zero unless the following \emph{selection rules} are satisfied\footnote{\url{https://dlmf.nist.gov/34.2}}: 
\begin{enumerate}
\item $m_i \in \{-\ell_i, -\ell_i + 1, -\ell_i + 2, \ldots, \ell_i\}, \quad (i = 1, 2, 3)$.
\item $m_1 + m_2 + m_3 = 0$
\item $|\ell_1 - \ell_2| \le \ell_3 \le \ell_1 + \ell_2$ 
\item $(\ell_1 + \ell_2 + \ell_3)$ is an integer (and, moreover, an even integer if  $m_1 = m_2 = m_3 = 0$).
\end{enumerate} 
The Wigner 3-j symbol  also satisfies 
\begin{equation} \label{e:time-reversal}
\begin{pmatrix}
  \ell_1 & \ell_2 & \ell_3\\
  -m_1 & -m_2 & -m_3
\end{pmatrix}
= (-1)^{\ell_1 + \ell_2 + \ell_3}
\begin{pmatrix}
  \ell_1 & \ell_2 & \ell_3\\
  m_1 & m_2 & m_3
\end{pmatrix}.
\end{equation}
The following Lemma gives an expression for $W^{p,k}_{q,m,n}$.

\begin{lem} \label{l:Wigner}
Let $p,k \geq0$ be fixed and let $|q|\leq p$ and $|m|,|n| \leq k$. Write $C_{p,k} = (2k+1) \frac{\sqrt{2p + 1}}{\sqrt{4 \pi}}$. 
Since $W^{p,k}_{q,m,n}$ is symmetric in $m$ and $n$, without loss of generality, assume that $m\geq n$. The integral defining $W^{p,k}_{q,m,n}$ in \eqref{e:W} can be expressed in terms of Wigner 3-j symbols as follows. 

\medskip
\noindent \underline{Case 1: $m > 0$ and $n >0$.}
\begin{equation*}
W^{p,k}_{q,m,n} = 
\begin{cases} 
0 & q < 0 \\
\delta_{m,n} C_{p,k} (-1)^m 
\begin{pmatrix}
p & k & k \\
0 & 0 & 0
\end{pmatrix}
\begin{pmatrix}
p & k & k \\
0 & m & -m
\end{pmatrix}
& q = 0 \\
\left\{ \begin{array}{ll}
\frac{1}{\sqrt{2}} C_{p,k} (-1)^{q}
\begin{pmatrix}
p & k & k \\
0 & 0 & 0
\end{pmatrix}
\begin{pmatrix}
p & k & k \\
-q & m & n
\end{pmatrix}
& \textrm{if } q = m + n \\
\frac{1}{\sqrt{2}} C_{p,k} (-1)^{m}
\begin{pmatrix}
p & k & k \\
0 & 0 & 0
\end{pmatrix}
\begin{pmatrix}
p & k & k \\
q & -m & n
\end{pmatrix}
 & \textrm{if } q = m - n \\
 0 & \textrm{otherwise}
\end{array}  \right\}  & q > 0
\end{cases}
\end{equation*}

\medskip
\noindent \underline{Case 2: $m =n= 0$.}
\begin{equation*}
W^{p,k}_{q,m,n} = 
\begin{cases} 
0 & q < 0 \\
C_{p,k} 
\begin{pmatrix}
p & k & k \\
0 & 0 & 0
\end{pmatrix}^2
& q = 0 \\
0  & q > 0
\end{cases}
\end{equation*}

\medskip
\noindent \underline{Case 3: $m < 0$ and $n < 0$.}
\begin{equation*}
W^{p,k}_{q,m,n} = 
\begin{cases} 
0 & q < 0 \\
\delta_{m,n} C_{p,k} (-1)^m 
\begin{pmatrix}
p & k & k \\
0 & 0 & 0
\end{pmatrix}
\begin{pmatrix}
p & k & k \\
0 & m & -m
\end{pmatrix}
& q = 0 \\
\left\{ \begin{array}{ll}
\frac{1}{\sqrt{2}} C_{p,k} (-1)^{q+1}
\begin{pmatrix}
p & k & k \\
0 & 0 & 0
\end{pmatrix}
\begin{pmatrix}
p & k & k \\
q & m & n
\end{pmatrix}
& \textrm{if } q = -m - n \\
\frac{1}{\sqrt{2}} C_{p,k} (-1)^{n}
\begin{pmatrix}
p & k & k \\
0 & 0 & 0
\end{pmatrix}
\begin{pmatrix}
p & k & k \\
q & -m & n
\end{pmatrix}
 & \textrm{if } q = m - n \\
 0 & \textrm{otherwise}
\end{array}  \right\}  & q > 0
\end{cases}
\end{equation*}

\medskip
\noindent \underline{Case 4: $m > 0$ and $n < 0$.}
\begin{equation*}
W^{p,k}_{q,m,n} = 
\begin{cases} 
\left\{ \begin{array}{ll}
\frac{1}{\sqrt{2}} C_{p,k} (-1)^{n}
\begin{pmatrix}
p & k & k \\
0 & 0 & 0
\end{pmatrix}
\begin{pmatrix}
p & k & k \\
q & -m & -n
\end{pmatrix}
& \textrm{if } q = m + n \\
\frac{1}{\sqrt{2}} C_{p,k} (-1)^{q}
\begin{pmatrix}
p & k & k \\
0 & 0 & 0
\end{pmatrix}
\begin{pmatrix}
p & k & k \\
q & m & -n
\end{pmatrix}
 & \textrm{if } q = n-m \\
\frac{1}{\sqrt{2}} C_{p,k} (-1)^{m+1}
\begin{pmatrix}
p & k & k \\
0 & 0 & 0
\end{pmatrix}
\begin{pmatrix}
p & k & k \\
q & m & n
\end{pmatrix}
 & \textrm{if } q = -n-m \\
 0 & \textrm{otherwise}
\end{array}  \right\} & q < 0 \\
0 & q = 0 \\
0 & q > 0
\end{cases}
\end{equation*}

\medskip
\noindent \underline{Case 5: $m > 0$ and $n = 0$.}
\begin{equation*}
W^{p,k}_{q,m,n} = 
\begin{cases} 
0 & q < 0 \\
0 & q = 0 \\
\delta_{q,m} C_{p,k} (-1)^q 
\begin{pmatrix}
p & k & k \\
0 & 0 & 0
\end{pmatrix}
\begin{pmatrix}
p & k & k \\
m & -m & 0
\end{pmatrix} 
& q > 0
\end{cases}
\end{equation*}

\medskip
\noindent \underline{Case 6: $m = 0$ and $n < 0$.}
\begin{equation*}
W^{p,k}_{q,m,n} = 
\begin{cases} 
\delta_{q,n} C_{p,k} (-1)^q 
\begin{pmatrix}
p & k & k \\
0 & 0 & 0
\end{pmatrix}
\begin{pmatrix}
p & k & k \\
n & -n & 0
\end{pmatrix} 
& q < 0 \\
0 & q = 0 \\
 0 & q > 0
\end{cases}
\end{equation*}
\end{lem}
A proof of Lemma~\ref{l:Wigner} is given in Appendix~\ref{s:Wigner}.

\subsection{An asymptotic expansion for Steklov eigenvalues} \label{s:AsymSteklov}

In \eqref{e:Lam1}, we have shown that the first-order perturbation of the $k^2, \ldots, (k+1)^2 -1$ Steklov eigenvalues are given by the $2k+1$ eigenvalues of the matrix $M^{(k)}$ given in \eqref{e:M}. The expression for $M^{(k)}$ involves 
the terms $W^{p,k}_{q,m,n}$ defined in \eqref{e:W} and computed in Lemma~\ref{l:Wigner}. 

All terms in $M^{(k)}$ involve
$\begin{pmatrix}
  p & k & k \\
  0 & 0 & 0
\end{pmatrix}$, 
which by the fourth selection rule is zero unless $p \in 2 \mathbb N$. 
Furthermore, by the third selection rule, 
we may assume that $ p \leq 2k$. Thus, we obtain 
\begin{equation} \label{e:M2}
M^{(k)}_{m,n}  
= - \frac{1}{2} \sum_{\substack{p=0 \\ p \textrm{ even}}}^{2k}\sum_{q = -p}^p A_{p,q}  \left( p(p+1) + 2k \right)  W^{p,k}_{q,m,n}. 
\end{equation}
We have show that $M^{(k)}$ is a symmetric matrix computed a finite sum.
By the spectral decomposition theorem for real symmetric matrices, there are $2k+1$ real eigenvalues in \eqref{e:Lam1} and the corresponding eigenvectors can be chosen to be orthogonal. 
Labelling each eigenvalue/eigenvector of $M^{(k)}$ with the subscript $n=-k,\ldots, k$,  we have 
\begin{equation} \label{e:Meigs}
M^{(k)} \alpha_n = \lambda_{k,n}^{(1)} \alpha_n, 
\qquad \qquad
n \leq |k|.
\end{equation}
If $\alpha_n$ is an eigenvector, the corresponding $O(\varepsilon^0)$ Steklov eigenfunction is given by 
\begin{equation} \label{e:StekEigFunPert}
u_{k,n}^{(0)} (r,\theta,\phi) = \sum_{m=-k}^k  (\alpha_n)_m r^{k} Y_{k,m}(\theta, \phi). 
\end{equation}
We summarize the analyticity result in \cite{Viator2019} and the preceding results in the following theorem.
\begin{thm} 
\label{t:StekAsym}
Fix $k \in \mathbb N$. 
The Steklov eigenvalues, $\lambda_n(\varepsilon)$, for $n \in \{ k^2, \ \ldots, \ (k+1)^2 - 1 \}$ consist of at most $2k+1$ branches of analytic functions which have at most algebraic singularities near $\varepsilon=0$.  At first-order in $\varepsilon$, the perturbation is given by the eigenvalues of the symmetric matrix $M^{(k)}$ in \eqref{e:M2}, as in \eqref{e:Meigs}. 
\end{thm}

\begin{cor} \label{c:OneCoef}
Consider a domain $\Omega_\varepsilon$ of the form in \eqref{e:dom} with 
$A_{p',q'} = \delta_{p',p} \delta_{q',q}$.  
We make the following general observations. 
\begin{enumerate}
\item If $p$ is odd, no eigenvalue is perturbed at $O(\varepsilon)$. 
\item If $p> 2k$, no eigenvalue is perturbed at $O(\varepsilon)$.
\item If  $p=q=0$, then 
\begin{equation} \label{e:Perp-k0}
\lambda^{(1)}_{k,n} = - \frac{k}{\sqrt{4 \pi}} \qquad \forall k \in \N, \ |n| \leq k. 
\end{equation}
\end{enumerate}
\end{cor}

\begin{proof}[Proof of Corollary~\ref{c:OneCoef}.]
(1) and (2) follows from the fact that $p$ odd and $p>2k$ does not make an appearance in \eqref{e:M2}. 
 
 For (3), we consider the case with $p=q=0$. 
 In this case, we have from \eqref{e:M2} that 
 $$
 M^{(k)}_{m,n} = - k W^{0,k}_{0,m,n}.
 $$
 From Lemma~\eqref{l:Wigner}, $M^{(k)}$ is a diagonal matrix. 
For $m<0$, we compute 
\begin{align*}
 M^{(k)}_{m,m} & = - k W^{0,k}_{0,m,m} \\
 &= - k C_{0,k} (-1)^m
 \begin{pmatrix}
  0 & k & k \\
  0 & 0 & 0
\end{pmatrix}
\begin{pmatrix}
  0 & k & k \\
  0 & m & -m
\end{pmatrix}. 
\end{align*}
Using the following identities\footnote{\url{https://dlmf.nist.gov/34.3}}
$$
\begin{pmatrix}
  0 & k & k \\
  0 & 0 & 0
\end{pmatrix} 
= (-1)^{k} \sqrt{ \frac{1}{1+ 2k}}
\quad \textrm{and} \quad
\begin{pmatrix}
  0 & k & k \\
  0 & m & -m
\end{pmatrix} 
= (-1)^{k-m} \sqrt{ \frac{1}{1+ 2k}}, 
$$
we obtain 
$$
M^{(k)}_{m,m} = - \frac{k}{\sqrt{4\pi}}. 
$$
A similar expression gives the same result for $m=0$ and $m>0$. 
The eigenvalues of this diagonal matrix are  $- \frac{k}{\sqrt{4\pi}}$ with mulitplicity $2k+1$, which gives the desired result.
\end{proof}

We interpret point (2) in Corollary~\ref{c:OneCoef} to mean that high frequency oscillations in the domain do not perturb low eigenvalues.

It is not difficult to show that the quantity $\Lambda(\Omega ) := \lambda(\Omega) \cdot |\Omega|^{\frac 1 3}$ is invariant to homothety, {\it i.e.}, 
$$
\Lambda(\alpha \Omega ) = \Lambda(\Omega ) \qquad \qquad \alpha > 0.
$$
Theorem~\ref{t:StekAsym} and Corollary~\ref{c:OneCoef} can be used to show the following local version of this statement. 
\begin{cor} \label{l:LocHom}
 $\Lambda(\Omega )$ is invariant to homothety for nearly circular $\Omega$.
\end{cor}
\begin{proof}
Fix $k \in \mathbb N$ and $|n| \leq k$. We consider a domain $\Omega_\varepsilon$ of the form in \eqref{e:dom} with $A_{\ell,m} = \delta_{\ell,0} \delta_{m,0}$.   We define
$$
\Lambda_{k,n}(\varepsilon) = \lambda_{k,n}(\varepsilon) |\Omega_\varepsilon|^{\frac 1 3}
$$
 and compute 
$$
 \frac{d}{d\varepsilon} \Lambda_{k,n}(0) 
  = \left( \frac{4 \pi }{3} \right)^{\frac 1 3} \lambda_{k,n}'(0) 
 + \left( \frac{1}{3} \right)^{\frac 1 3} \left( \frac{1}{4 \pi} \right)^{\frac 1 6}k. 
$$
Using \eqref{e:Perp-k0}, we find that $ \frac{d}{d\varepsilon} \Lambda_{k,n}(0)  = 0$, as desired.
\end{proof}

\section{Proof of Theorem \ref{t:stationary}} \label{s:stationary} 
We now prove Theorem \ref{t:stationary}.  Let $k^2 \in \mathbb{N}$ be fixed.  We will show that $\lambda_{k^2}^{(1)} \leq 0$. By Theorem \ref{t:StekAsym}, we know that $\lambda_{k^2}^{(1)}$ is the smallest eigenvalue of the matrix $M^{(k)}$.  We will show that $\text{tr}(M^{(k)}) = 0$.
Now
\begin{align*}
    \text{tr}(M^{(k)}) = \sum \limits_{m=-k}^k M^{(k)}_{m,m} = - \frac{1}{2} \sum_{\substack{p=0 \\ p \textrm{ even}}}^{2k}\sum_{q = -p}^p A_{p,q}  \left( p(p+1) + 2k \right) \sum \limits_{m=-k}^k  W^{p,k}_{q,m,m}. 
\end{align*}
We first note that, along the diagonal of the matrix $M^{(k)}$, we have $W^{p,k}_{q,m,m} = 0$ for $q<0$ by Lemma \ref{l:Wigner}.  Thus, if we call $\Theta_q^{p,k} = \sum \limits_{m=-k}^k  W^{p,k}_{q,m,m}$, the trace of $M^{(k)}$ reduces to
\begin{align*}
    \text{tr}(M^{(k)}) =  - \frac{1}{2} \sum_{\substack{p=0 \\ p \textrm{ even}}}^{2k}\sum_{q = 0}^p A_{p,q}  \left( p(p+1) + 2k \right) \Theta_q^{p,k}. 
\end{align*}
Set $q=0$.  Then, again using Lemma \ref{l:Wigner}, we compute
\begin{align*}
    \Theta_q^{p,k} = \sum \limits_{m=-k}^k  W^{p,k}_{q,m,m} =  C_{p,k} \begin{pmatrix}
p & k & k \\
0 & 0 & 0
\end{pmatrix} 
\sum \limits_{m=-k}^k (-1)^m 
\begin{pmatrix}
p & k & k \\
0 & m & -m
\end{pmatrix}
=0,
\end{align*}
where we have used a Wigner 3-$j$ symbol identity \cite{WolframWigner}. 

Now suppose $q>0$.  Using Lemma \ref{l:Wigner} one more time, we obtain
\begin{align*}
    \Theta_q^{p,k} & = \sum \limits_{m=-k}^k  W^{p,k}_{q,m,m} =  \sum \limits_{m=-k}^{-1}  W^{p,k}_{q,m,m} + \sum \limits_{m=1}^k  W^{p,k}_{q,m,m} \\
     & = \sum \limits_{m=1}^k \frac{1}{\sqrt{2}} C_{p,k} (-1)^{q}
\begin{pmatrix}
p & k & k \\
0 & 0 & 0
\end{pmatrix}
\begin{pmatrix}
p & k & k \\
-q & m & m
\end{pmatrix}
-\sum \limits_{m=-k}^{-1} \frac{1}{\sqrt{2}}C_{p,k} (-1)^{q}
\begin{pmatrix}
p & k & k \\
0 & 0 & 0
\end{pmatrix}
\begin{pmatrix}
p & k & k \\
q & m & m
\end{pmatrix} \\
& = \sum \limits_{m=1}^k \frac{1}{\sqrt{2}} C_{p,k} (-1)^{q}
\begin{pmatrix}
p & k & k \\
0 & 0 & 0
\end{pmatrix}
\begin{pmatrix}
p & k & k \\
-q & m & m
\end{pmatrix}
- \frac{1}{\sqrt{2}}C_{p,k} (-1)^{q}
\begin{pmatrix}
p & k & k \\
0 & 0 & 0
\end{pmatrix}
\begin{pmatrix}
p & k & k \\
-q & m & m
\end{pmatrix} = 0
\end{align*}
by \eqref{e:time-reversal}.  We conclude that $\text{tr}(M^{(k)}) = 0$, and thus $\lambda_{k^2}^{(1)} \leq 0$ as desired.  This completes the proof.

\clearpage
\printbibliography

\clearpage
\appendix
\section{Proof of Lemma~\ref{l:Wigner}} \label{s:Wigner}
We use the expression for real spherical harmonics in terms of complex spherical harmonics in \eqref{e:RealHarmonicsa} and use \eqref{e:trip} to evaluate the six cases for the combinations of ($n < 0$, $n=0$, and $n>0$) and ($m< 0$, $m=0$, and $m>0$) assuming $m \geq n$ in turn. 

\medskip
\noindent \underline{Case 1: $m > 0$ and $n >0$.}
For $q < 0$, we compute 
\begin{align*}
W^{p,k}_{q,m,n} = \frac{i}{2 \sqrt{2}} \iint_{S^2} (Y_p^q - (-1)^q Y_p^{-q} ) (Y_k^{-m} + (-1)^m Y_k^m ) (Y_k^{-n} + (-1)^n Y_k^n ) \ dS
= 0
\end{align*}
since the integral of a triple product of complex spherical harmonics is a real number and so is $W^{p,k}_{q,m,n}$.

For $q=0$, we have
\begin{align*}
W^{p,k}_{q,m,n} & = \frac{1}{2} \iint_{S^2} Y_p^0 (Y_k^{-m} + (-1)^m Y_k^m ) (Y_k^{-n} + (-1)^n Y_k^n ) \ dS \\
& = \frac{1}{2} \iint_{S^2} Y_p^0 [Y_k^{-m}Y_k^{-n} + (-1)^m Y_k^mY_k^{-n} +(-1)^nY_k^{-m}Y_k^n + (-1)^{n+m}Y_k^mY_k^n] \ dS
\end{align*}
Applying \eqref{e:trip} and the selection criteria for Wigner 3-$j$ symbols, we obtain
\begin{align*}
  W^{p,k}_{q,m,n}=  \delta_{m,n} C_{p,k} (-1)^m 
\begin{pmatrix}
p & k & k \\
0 & 0 & 0
\end{pmatrix}
\begin{pmatrix}
p & k & k \\
0 & m & -m
\end{pmatrix}
\end{align*}

For $q > 0$, we compute
{\tiny
\begin{align*}
W^{p,k}_{q,m,n} & = \frac{1}{2 \sqrt{2}} \iint_{S^2} (Y_p^{-q} + (-1)^q Y_p^{q} ) (Y_k^{-m} + (-1)^m Y_k^m ) (Y_k^{-n} + (-1)^n Y_k^n ) \ dS \\
& = \frac{1}{2 \sqrt{2}} \iint_{S^2} 
Y_p^{-q} Y_k^{-m} Y_k^{-n} + (-1)^n Y_p^{-q} Y_k^{-m} Y_k^{n} 
+ (-1)^m Y_p^{-q} Y_k^{m} Y_k^{-n} + (-1)^{m+n} Y_p^{-q} Y_k^{m} Y_k^{n} \\
& \qquad \qquad \quad  +  (-1)^q Y_p^{q} Y_k^{-m} Y_k^{-n}  + (-1)^{q+n} Y_p^{q} Y_k^{-m} Y_k^{n}  
+ (-1)^{q+m} Y_p^{q} Y_k^{m} Y_k^{-n} + (-1)^{q + m+n} Y_p^{q} Y_k^{m} Y_k^{n} \ dS \\
& = \frac{1}{2 \sqrt{2}} C_{p,k}
\begin{pmatrix}
p & k & k \\
0 & 0 & 0
\end{pmatrix}
\Big[ 
\begin{pmatrix}
p & k & k \\
-q & -m & -n
\end{pmatrix}
+ (-1)^n 
\begin{pmatrix}
p & k & k \\
-q & -m & n
\end{pmatrix}
+ (-1)^m 
\begin{pmatrix}
p & k & k \\
-q & m & -n
\end{pmatrix}
+ (-1)^{m+n} 
\begin{pmatrix}
p & k & k \\
-q & m & n
\end{pmatrix}  \\
& \qquad \qquad \qquad \quad  
+ (-1)^q 
\begin{pmatrix}
p & k & k \\
q & -m & -n
\end{pmatrix}
+ (-1)^{q+n} 
\begin{pmatrix}
p & k & k \\
q & -m & n
\end{pmatrix}
+ (-1)^{q+m} 
\begin{pmatrix}
p & k & k \\
q & m & -n
\end{pmatrix}
+ (-1)^{q + m+n} 
\begin{pmatrix}
p & k & k \\
q & m & n
\end{pmatrix}
\Big]. 
\end{align*}}
Since $q,m,n>0$ the first and last Wigner 3-j symbols are zero by the second selection rule. 
Since $m\geq n$, by the second selection rule, we also have that the second and second to last terms vanish. 
Furthermore, using \eqref{e:time-reversal}, each of the remaining sums combine to give
\begin{align*}
W^{p,k}_{q,m,n} & = \frac{1}{ \sqrt{2}} C_{p,k}
\begin{pmatrix}
p & k & k \\
0 & 0 & 0
\end{pmatrix}
\left[ (-1)^m 
\begin{pmatrix}
p & k & k \\
q & -m & n
\end{pmatrix}
+ (-1)^q 
\begin{pmatrix}
p & k & k \\
-q & m & n
\end{pmatrix}
\right] \\
&= \left\{ \begin{array}{ll}
\frac{1}{\sqrt{2}} C_{p,k} (-1)^{q}
\begin{pmatrix}
p & k & k \\
0 & 0 & 0
\end{pmatrix}
\begin{pmatrix}
p & k & k \\
-q & m & n
\end{pmatrix}
& \textrm{if } q = m + n \\
\frac{1}{\sqrt{2}} C_{p,k} (-1)^{m}
\begin{pmatrix}
p & k & k \\
0 & 0 & 0
\end{pmatrix}
\begin{pmatrix}
p & k & k \\
q & -m & n
\end{pmatrix}
 & \textrm{if } q = m - n \\
 0 & \textrm{otherwise}
\end{array} \right. ,
\end{align*}
as desired.

\bigskip
\noindent \underline{Case 2: $m =n= 0$.}
For $q<0$, we have
\begin{align*}
    W^{p,k}_{q,m,n} & = \frac{i}{\sqrt{2}} \iint_{S^2} (Y_p^q - (-1)^q Y_p^{-q} ) (Y_k^{0})^2 \ dS
= 0
\end{align*}
as in Case 1.

For $q=0$, we have
\begin{align*}
    W^{p,k}_{q,m,n} & =  \iint_{S^2} Y_p^0 (Y_k^{0})^2 \ dS 
= C_{p,k} 
\begin{pmatrix}
p & k & k \\
0 & 0 & 0
\end{pmatrix}^2
\end{align*}
by \eqref{e:trip}.

Finally, for $q>0$, we obtain
\begin{align*}
W^{p,k}_{q,m,n} & = \frac{1}{\sqrt{2}} \iint_{S^2} (Y_p^{-q} + (-1)^q Y_p^{q} ) (Y_k^{0})^2 \ dS \\
&= \frac{1}{\sqrt{2}} C_{p,k} 
\begin{pmatrix}
p & k & k \\
0 & 0 & 0
\end{pmatrix}
\left[ 
\begin{pmatrix}
p & k & k \\
-q & 0 & 0
\end{pmatrix}
+ (-1)^q
\begin{pmatrix}
p & k & k \\
q & 0 & 0
\end{pmatrix}
\right]  \\
&= 0
\end{align*}
by the second selection rule.

\bigskip
\noindent \underline{Case 3: $m < 0$ and $n < 0$.}
If $q<0$, then we calculate
\begin{align*}
W^{p,k}_{q,m,n} = \frac{-i}{2 \sqrt{2}} \iint_{S^2} (Y_p^q - (-1)^q Y_p^{-q} ) (Y_k^{m} - (-1)^m Y_k^{-m} ) (Y_k^{n} - (-1)^n Y_k^{-n} ) \ dS
= 0
\end{align*}
as in Case 1.

If $q=0$, then we have
\begin{align*}
W^{p,k}_{q,m,n} = -\frac{1}{2} \iint_{S^2} Y_p^0 [Y_k^{m}Y_k^{n} - (-1)^m Y_k^{-m}Y_k^{n} -(-1)^nY_k^{m}Y_k^{-n} + (-1)^{n+m}Y_k^{-m}Y_k^{-n}] \ dS
\end{align*}
Applying \eqref{e:trip} and the selection criteria for Wigner 3-$j$ symbols, we obtain
\begin{align*}
W^{p,k}_{q,m,n} = \delta_{m,n} C_{p,k} (-1)^m 
\begin{pmatrix}
p & k & k \\
0 & 0 & 0
\end{pmatrix}
\begin{pmatrix}
p & k & k \\
0 & m & -m
\end{pmatrix}
\end{align*}
as desired.

For $q>0$, we compute
\begin{align*}
W^{p,k}_{q,m,n} & = -\frac{1}{2\sqrt{2}} \iint_{S^2} [ Y_p^{-q}Y_k^{m}Y_k^{n} -(-1)^nY_p^{-q}Y_k^{m}Y_k^{-n} - (-1)^mY_p^{-q}Y_k^{-m}Y_k^{n} + (-1)^{n+m}Y_p^{-q}Y_k^{-m}Y_k^{-n} \\
& + (-1)^qY_p^{q}Y_k^{m}Y_k^{n} - (-1)^{n+q}Y_p^qY_k^{m}Y_k^{-n}-(-1)^{m+q}Y_p^qY_k^{-m}Y_k^{n} + (-1)^{n+m+q}Y_p^qY_k^{-m}Y_k^{-n}] \ dS
\end{align*}
Since $m,n<0$ and $q>0$, the first and last terms vanish.  If $q \neq m-n$ and  $q\neq -n-m$, then by the selection criteria all of the above terms are zero.  If $q=m-n$, then only the second and second-to-last terms are non-zero, and we obtain
\begin{align*}
W^{p,k}_{q,m,n} & = \frac{1}{\sqrt{2}} C_{p,k} (-1)^{n}
\begin{pmatrix}
p & k & k \\
0 & 0 & 0
\end{pmatrix}
\begin{pmatrix}
p & k & k \\
q & -m & n
\end{pmatrix}
\end{align*}

Finally, if $q=-n-m$, then only the fourth and fifth terms are non-zero, and we have
\begin{align*}
W^{p,k}_{q,m,n} & = \frac{1}{\sqrt{2}} C_{p,k} (-1)^{q+1}
\begin{pmatrix}
p & k & k \\
0 & 0 & 0
\end{pmatrix}
\begin{pmatrix}
p & k & k \\
q & m & n
\end{pmatrix}
\end{align*}
as desired.

\bigskip
\noindent \underline{Case 4: $m > 0$ and $n < 0$.}

For $q<0$, we have,
\begin{align*}
W^{p,k}_{q,m,n} & = -\frac{1}{2\sqrt{2}} \iint_{S^2} [ Y_p^{q}Y_k^{-m}Y_k^{n} -(-1)^nY_p^{q}Y_k^{-m}Y_k^{-n} + (-1)^mY_p^{q}Y_k^{m}Y_k^{n} - (-1)^{n+m}Y_p^{q}Y_k^{m}Y_k^{-n} \\
& - (-1)^qY_p^{-q}Y_k^{-m}Y_k^{n} + (-1)^{n+q}Y_p^{-q}Y_k^{-m}Y_k^{-n}-(-1)^{m+q}Y_p^{-q}Y_k^{m}Y_k^{n} + (-1)^{n+m+q}Y_p^{-q}Y_k^{m}Y_k^{-n}] \ dS
\end{align*}
Since $m>0$ and $q,n<0$, the first and last terms vanish.  If $q \notin \{m+n, n-m, -n-m \}$, then the rest of the terms vanish as well by the selection criteria.  If $q = n+m$, then only the second and second-to-last terms are nonzero, and we obtain
\begin{align*}
W^{p,k}_{q,m,n} & = \frac{1}{\sqrt{2}} C_{p,k} (-1)^{n}
\begin{pmatrix}
p & k & k \\
0 & 0 & 0
\end{pmatrix}
\begin{pmatrix}
p & k & k \\
q & -m & -n
\end{pmatrix}
\end{align*}
If $q=n-m$, then only the fourth and fifth terms are nonzero, and we obtain
\begin{align*}
W^{p,k}_{q,m,n} & = \frac{1}{\sqrt{2}} C_{p,k} (-1)^{q}
\begin{pmatrix}
p & k & k \\
0 & 0 & 0
\end{pmatrix}
\begin{pmatrix}
p & k & k \\
q & m & -n
\end{pmatrix}
\end{align*}
Finally, if $q=-n-m$, the only the third and sixth terms are nonzero, and we obtain
\begin{align*}
W^{p,k}_{q,m,n} & = \frac{1}{\sqrt{2}} C_{p,k} (-1)^{m+1}
\begin{pmatrix}
p & k & k \\
0 & 0 & 0
\end{pmatrix}
\begin{pmatrix}
p & k & k \\
q & m & n
\end{pmatrix}
\end{align*}

For $q=0$, we have
\begin{align*}
W^{p,k}_{q,m,n} = \frac{i}{2} \iint_{S^2} Y_p^0 (Y_k^{-m} + (-1)^m Y_k^m ) (Y_k^{n} - (-1)^n Y_k^{-n} ) \ dS
= 0
\end{align*}
as in Case 1. 

For $q>0$, we have 
\begin{align*}
W^{p,k}_{q,m,n} = \frac{i}{2\sqrt{2}} \iint_{S^2} (Y_p^{-q} + (-1)^q Y_p^q ) (Y_k^{-m} + (-1)^m Y_k^m ) (Y_k^{n} - (-1)^n Y_k^{-n} ) \ dS
= 0
\end{align*}
as in Case 1.

\bigskip
\noindent \underline{Case 5: $m > 0$ and $n = 0$.}

For $q<0$, we have
\begin{align*}
W^{p,k}_{q,m,n} = \frac{i}{2} \iint_{S^2} (Y_p^{q} - (-1)^q Y_p^{-q} ) (Y_k^{-m} + (-1)^m Y_k^m )  Y_k^0 \ dS
= 0
\end{align*}
as in Case 1. 

For $q=0$, we have
\begin{align*}
W^{p,k}_{q,m,n} & = \frac{1}{\sqrt{2}} \iint_{S^2} Y_p^0 (Y_k^{-m} + (-1)^m Y_k^{m}) Y_k^{0} \ dS \\
&= \frac{1}{\sqrt{2}} C_{p,k} 
\begin{pmatrix}
p & k & k \\
0 & 0 & 0
\end{pmatrix}
\left[ 
\begin{pmatrix}
p & k & k \\
0 & -m & 0
\end{pmatrix}
+ (-1)^m
\begin{pmatrix}
p & k & k \\
0 & m & 0
\end{pmatrix}
\right]  \\
&= 0
\end{align*}
by the second selection rule.

For $q>0$, we have
{\tiny
\begin{align*}
W^{p,k}_{q,m,n} 
&= \frac{1}{2} \iint_{S^2} (Y_p^{-q} + (-1)^q Y_p^{q}) (Y_k^{-m} + (-1)^m Y_k^{m}) Y_k^0 \ dS \\
&= \frac{1}{2} \iint_{S^2} (Y_p^{-q}Y_k^{-m} Y_k^0 + (-1)^m Y_p^{-q} Y_k^{m} Y_k^0
+ (-1)^q Y_p^{q} Y_k^{-m} Y_k^0 + (-1)^{m+q} Y_p^{q} Y_k^{m} Y_k^0 \ dS \\
&= \frac{1}{2} C_{p,k}
\begin{pmatrix}
p & k & k \\
0 & 0 & 0
\end{pmatrix}
\left[
\begin{pmatrix}
p & k & k \\
-q & -m & 0
\end{pmatrix}
+ (-1)^m 
\begin{pmatrix}
p & k & k \\
-q & m & 0
\end{pmatrix}
+ (-1)^q 
\begin{pmatrix}
p & k & k \\
q & -m & 0
\end{pmatrix}
+ (-1)^{m+q} 
\begin{pmatrix}
p & k & k \\
q & m & 0
\end{pmatrix}
\right].
\end{align*}}
Using the selection rules, we obtain 
\begin{align*}
W^{p,k}_{q,m,n} =
\delta_{q,m} C_{p,k} (-1)^q 
\begin{pmatrix}
p & k & k \\
0 & 0 & 0
\end{pmatrix}
\begin{pmatrix}
p & k & k \\
m & -m & 0
\end{pmatrix},  
\end{align*}
as desired.

\bigskip
\noindent \underline{Case 6: $m = 0$ and $n < 0$.}
If $q<0$, then we obtain
\begin{align*}
W^{p,k}_{q,m,n} = -\frac{1}{2} \iint_{S^2} Y_k^0[Y_p^{q}Y_k^n - (-1)^nY_p^{q}Y_k^{-n} - (-1)^qY_p^{-q}Y_k^{n}+(-1)^{q+n}Y_p^{-q}Y_k^{-n}] \ dS
\end{align*}
The first and last terms vanish since $q,n<0$.  If $q \neq n$, then the middle two terms vanish as well, and so we obtain
\begin{align*}
   W^{p,k}_{q,m,n} = \delta_{q,n} C_{p,k} (-1)^q 
\begin{pmatrix}
p & k & k \\
0 & 0 & 0
\end{pmatrix}
\begin{pmatrix}
p & k & k \\
n & -n & 0
\end{pmatrix} 
\end{align*}
If $q=0$, then
\begin{align*}
W^{p,k}_{q,m,n} = -\frac{i}{\sqrt{2}} \iint_{S^2} Y_k^0Y_p^{0}[Y_k^n - (-1)^nY_k^{-n}] \ dS = 0
\end{align*}
as in Case 1.  Finally, if $q>0$, then we have
\begin{align*}
W^{p,k}_{q,m,n} = -\frac{i}{2} \iint_{S^2} Y_k^0[Y_p^{-q}Y_k^n - (-1)^nY_p^{-q}Y_k^{-n} + (-1)^qY_p^{q}Y_k^{n}-(-1)^{q+n}Y_p^{q}Y_k^{-n}] \ dS = 0
\end{align*}
as in Case 1.
\end{document}